\documentclass[reqno]{amsart} 
\usepackage{amsmath,amsfonts,amssymb}
\usepackage{graphicx}

\author{Poo-Sung Park}
\address{Department of Mathematics Education\\ 
Kyungnam University\\
Changwon, 51767 \\
Republic of Korea}
\email{pspark@kyungnam.ac.kr}

\title[Multiplicative functions additive on triangular numbers]{Multiplicative functions which are\\ additive on triangular numbers}

\keywords{identity function, multiplicative function, triangular numbers}
\subjclass[2010]{11P32, 11A25}

\theoremstyle{plain}

\newtheorem{thm}{Theorem}[section] 

\newtheorem{lem}[thm]{Lemma} 

\theoremstyle{definition}

\theoremstyle{remark}

\newcommand{\tn}{\triangle}
\newcommand{\ptn}{\triangle^{\hspace{-.6ex}\raisebox{.3ex}{\scalebox{.5}{\ensuremath +}}}}

\begin{document}

\thanks{This research was supported by Basic Science Research Program through the National Research Foundation of Korea(NRF) funded by the Ministry of  Science and ICT(NRF-2017R1A2B1010761).}

\begin{abstract}
Fix $k \ge 3$. If a multiplicative function $f$ satisfies 
\[
f(x_1+x_2+\dots+x_k) = f(x_1) + f(x_2) + \dots + f(x_k)
\]
for arbitrary positive triangular numbers $x_1, x_2, \dots, x_k$, then $f$ is the identity function. This extends Chung and Phong's work for $k=2$.
\end{abstract}

\maketitle

\section{Introduction}

Claudia Spiro's paper \cite{Spiro} in 1992 has inspired lots of mathematicians to produce many related papers. She showed that a multiplicative function $f$ satisfying $f(p+q) = f(p)+f(q)$ for arbitrary prime numbers $p$ and $q$ is the identity function under some condition. Let $E$ be a set of arithmetic functions and let $S$ be a set of positive integers. Spiro dubbed $S$ the \emph{additive uniqueness set} for $E$ if a function $f \in E$ is uniquely determined under the condition $f(a+b) = f(a)+f(b)$ for $a,b \in S$.

In 1999 Chung and Phong \cite{C-P} showed that the set of positive triangular numbers and the set of positive tetrahedral numbers are new additive uniqueness sets for multiplicative functions. They also conjectured that the set
\[
H_k = \left\{ \frac{n(n+1)\dots(n+k-1)}{1\cdot2\dots k} \,\bigg|\, n=1,2,3, \dots \right\}
\]
is an additive uniqueness set for every $k \ge 4$.

In 2010 Fang \cite{Fang} extended Spiro's work to the condition $f(p+q+r) = f(p)+f(q)+f(r)$ for arbitrary prime numbers $p, q, r$. His work was generalized by Dubickas and \v{S}arka \cite{D-S} to sums of arbitrary number of primes.

Let us consider the general condition \emph{$k$-additivity}. That is, if a function $f \in E$ satisfying $f(x_1+x_2+\dots+x_k) = f(x_1)+f(x_2)+\dots+f(x_k)$ for arbitrary $x_i \in S$ is uniquely determined, we call $S$ the $k$-additive uniqueness set for $E$. We can say that the set of prime numbers is the $k$-additive uniqueness set with $k \ge 2$.

Here is an interesting example. The set of nonzero squares for the set of multiplicative functions is not a $2$-additive uniqueness set \cite{Chung}, but is $k$-additive uniqueness set for every $k \ge 3$ \cite{Park}. So it is natural to ask whether a $2$-additive uniqueness set is also a $k$-additive uniqueness set or not for $k \ge 3$.

Let $\mathbb{T}$ be the set of triangular numbers $T_n = \frac{n(n+1)}{2}$ for $n \ge 1$. That is,
\[
\mathbb{T} = \{ 1, 3, 6, 10, 15, 21, 28, 36, 45, 55, \dots \}.
\]
In this article, we show that $\mathbb{T}$ is the $k$-additive uniqueness set for multiplicative functions. This extends Chung and Phong's work for $2$-additive uniqueness of $\mathbb{T}$. 

The proof consists of three parts. The first is about the $3$-additivity, the second is about the $4$-additivity, and the last is about the $k$-additivity with $k \ge 5$. For convenience, we denote a triangular number by $\tn$. If the triangular number is restricted to be positive, we use the symbol $\ptn$.

%
%

\section{$3$-additive uniqueness set}

\begin{thm}\label{thm:k=3}
If a multiplicative function $f$ satisfies 
\[
f(a+b+c) = f(a)+f(b)+f(c)
\]
for $a,b,c \in \mathbb{T}$, then $f$ is the identity function.
\end{thm}

Clearly, $f(1)=1$ and $f(3)=3$. Note that $f(5) = f(1+1+3) = 5$. The equalities
\begin{align*}
f(10)
&= f(1+3+6) = 4+3f(2) \\
&= f(2) \cdot f(5) = 5f(2)
\end{align*}
yields $f(2) = 2$. Then, $f(6) = 6$ and $f(10) = 10$.

We use induction. Suppose that $f(n) = n$ for all $n < N$. Now let us show $f(N)=N$. We may assume that $N = p^r$ for some prime $p$ by the multiplicity of $f$.

If $N = 3^r$, then from the equalities
\begin{align*}
f(3 T_{3^{r-1}}) 
&= 3f(T_{3^{r-1}}) = 3f\!\left(\frac{3^{r-1}(3^{r-1}+1)}{2}\right) = 3f(3^{r-1}) \cdot f\!\left( \frac{3^{r-1}+1}{2} \right) \\
&= f\!\left( 3^r \frac{3^{r-1}+1}{2} \right) = f(3^{r}) \cdot f\!\left( \frac{3^{r-1}+1}{2} \right) 
\end{align*}
we conclude that $f(3^r) = 3^r$ since $f\!\left( \frac{3^{r-1}+1}{2} \right) = \frac{3^{r-1}+1}{2}$ by induction hypothesis.

Now, assume that $N = p^r = 3s-1$ with odd prime $p$. Note that $f(T_{s-1}) = T_{s-1}$ and $f(T_s) = T_s$ by induction hypothesis since $T_s$ can be factored into integers smaller than $N$. Since
\begin{align*}
f(T_{s-1} + T_{s-1} + T_{s}) 
&= \frac{s(s-1)}{2} + \frac{s(s-1)}{2} + \frac{s(s+1)}{2} = \frac{s(3s-1)}{2} = \frac{s p^r}{2} \\
&= f\!\left( \frac{s(3s-1)}{2} \right) = f\!\left(\frac{s}{2}\right) \cdot f(3s-1) = \frac{s}{2} f(p^r),
\end{align*}
we know that $f(p^r) = p^r$.

If $N = p^r = 3s+1$ with odd prime $p$. Then, the equalities
\begin{align*}
f(T_{s-1} + T_s + T_s) 
&= \frac{s(s-1)}{2} + \frac{s(s+1)}{2} + \frac{s(s+1)}{2} = \frac{s(3s+1)}{2} = \frac{s p^r}{2} \\
&= f\!\left( \frac{s(3s+1)}{2} \right) = f\!\left(\frac{s}{2}\right) \cdot f(3s+1) = \frac{s}{2} f(p^r)
\end{align*}
show that $f(p^r) = p^r$.

Now, we consider the last case $N = 2^r$. Let $2^{r+1} = 3s \pm 1$. Then,
\begin{align*}
f(T_{s-1} + T_{s-1} + T_{s}) 
&= \frac{s(s-1)}{2} + \frac{s(s-1)}{2} + \frac{s(s+1)}{2} = \frac{s(3s-1)}{2} = s 2^r \\
&= f\!\left( \frac{s(3s-1)}{2} \right) = f(s) \cdot f\!\left(\frac{3s-1}{2}\right) = s f(2^r)
\intertext{or}
f(T_{s-1} + T_s + T_{s}) 
&= \frac{s(s-1)}{2} + \frac{s(s+1)}{2} + \frac{s(s+1)}{2} = \frac{s(3s+1)}{2} = s 2^r \\
&= f\!\left( \frac{s(3s+1)}{2} \right) = f(s) \cdot f\!\left(\frac{3s+1}{2}\right) = s f(2^r).
\end{align*}
Thus, $f(2^r) = 2^r$.

\section{$4$-additive uniqueness set}

\begin{thm}\label{thm:k=4}
If a multiplicative function $f$ satisfies 
\[
f(a+b+c+d) = f(a)+f(b)+f(c)+f(d)
\]
for $a,b,c,d \in \mathbb{T}$, then $f$ is the identity function.
\end{thm}

\begin{lem}\label{lem:sum_of_k_tn}
Let $\mathbb{T}_k$ be the set of sums of $k$ $\ptn$. If $k \ge 4$, then $\mathbb{T}_k$ is the set of all positive integers except for $1, 2 \dots, k-1, k+1, k+3$.
\end{lem}
\begin{proof}
Gauss' theorem guarantees that every positive integer can be written as $\tn + \tn + \tn$, some of which possibly vanish. Thus, if $n > 21$ is given, then $n-21$ is $\ptn$, $\ptn+\ptn$, or $\ptn+\ptn+\ptn$. But, since $n = (n-21)+21 = (n-21)+6+15 = (n-21)+3+3+15$, every integer $>21$ can be written as a sum of four $\ptn$. 

It can be easily verified that every positive integer $\le 21$ is a sum of four $\ptn$ except for $1, 2, 3, 5$, and $7$. Hence, we can conclude that every positive integer $\ge 8$ can be written as a sum of four $\ptn$.

Now, consider the general cases. It is clear that the sum of $k$ $\ptn$ can represent $k$ and $k+2$ but cannot represent from $1$ through $k-1$. It is also easily checked that the sum cannot represent $k+1$ and $k+3$. Since sums of four $\ptn$ represent all integers $\ge 8$, the sum
\[
\underbrace{1+\dots+1}_{k-4\text{ summands}} + \ptn + \ptn + \ptn + \ptn 
\]
represents all integers $\ge k+4$.
\end{proof}

Now let us prove Theorem \ref{thm:k=4}. Note that $f(1) = 1$ and $f(4) = 4$. From
\begin{align*}
f(6) 
&= f(1+1+1+3) = 3+f(3) \\
&= f(2) \cdot f(3) \\
f(10) 
&= f(1+3+3+3) = 1+3f(3) \\
&= f(2) \cdot f(5) \\
f(15) 
&= f(3+3+3+6) = 3f(3)+f(2) \cdot f(3) \\
&= f(3) \cdot f(5)
\end{align*}
we obtain that two solutions:
\[
f(2)=-2, f(3)=-1, f(5)=1\quad\text{or}\quad
f(2)=2, f(3)=3, f(5)=5.
\]

First case yields $f(9) = f(1+1+1+6) = 3+f(2) \cdot f(3) = 5$. But, this would make a contradiction:
\begin{align*}
f(18) 
&= f(1+1+1+15) = 3+f(3) \cdot f(5) = 2 \\
&= f(2) \cdot f(9) = -10.
\end{align*}
Thus, we can conclude that $f(2)=2$, $f(3)=3$, and $f(5)=5$. Then, $f(14) = f(1+1+6+6) = f(2) \cdot f(7)$ gives $f(7)=7$. So $f(n) = n$ for $n \le 7$.

By Lemma \ref{lem:sum_of_k_tn} every integer $\ge 8$ can be written as a sum of four $\ptn$. Thus $f$ should be the identity function by induction.

\section{$k$-additive uniqueness set}

Let $k \ge 5$. Note that
\begin{align*}
(k-2) + 16 
&= (k-2)\cdot1 + 6+10 \\
&= (k-2)\cdot1 + 1+15 \\
(k-3) + 12 
&= (k-3)\cdot1 + 3+3+6 \\
&= (k-3)\cdot1 + 1+1+10 \\
(k-4) + 19 
&= (k-4)\cdot1 + 1+6+6+6 \\
&= (k-4)\cdot1 + 3+3+3+10.
\end{align*}
Thus, the equalities give rise to the system of equations
\begin{align*}
&f(2) \cdot f(3)+f(2) \cdot f(5) = 1+f(3) \cdot f(5) \\
&2f(3)+f(2) \cdot f(3) = 2+f(2) \cdot f(5) \\
&1+3f(2) \cdot f(3) = 3f(3)+f(2) \cdot f(5).
\end{align*}
The solutions are
\begin{align*}
&f(2) = \frac14, \quad f(3) = \frac23, \quad f(5) = -2;\\
&f(2) = f(3) = f(5) = 1;\\
&f(2) = 2, \quad f(3) = 3, \quad f(5) = 5.
\end{align*}

Note that $f(k+2) = k-1+f(3)$ and $f(k+4) = k-2+2f(3)$.

If $3 \nmid (k+2)$, then the equalities
\begin{align*}
f(3(k+2))
&= f(\hspace{-.5ex}\underbrace{3+\dots+3}_{k-2\text{ summands}} \hspace{-.5ex} +\,\, 6+6) = f(3)(k-2)+2f(2) \cdot f(3) \\
&= f(3) \cdot f(k+2) = f(3) (k-1+f(3))
\end{align*}
exclude the first solution set $f(2) = \frac14, f(3) = \frac23, f(5) = -2$.

If $3 \mid (k+2)$, then we consider
\begin{align*}
f(3(k+4))
&= f(\hspace{-.5ex}\underbrace{3+\dots+3}_{k-1\text{ summands}} \hspace{-.5ex}+\, 15) = f(3)(k-1)+f(3) \cdot f(5) \\
&= f(3) \cdot f(k+4) = f(3) (k-2+2f(3)),
\end{align*}
which exclude the first solution set.

Now, consider the second solution set $f(2)=f(3)=f(5)=1$. Then, $f(T_1) = f(T_2) = f(T_3) = f(T_4) = f(T_5) = 1$. By Lemma \ref{lem:sum_of_k_tn} we have that every $T_n$ with $n \ge 4$ can be written as a sum of four $\ptn$. From the equality
\[
\tag{$\ast$}\label{eqn:T_s}
\begin{aligned}
&(k-5) + 1 + 1 + 1 + 3 + T_s \\
&= (k-5) + 6 + T_a + T_b + T_c + T_d \qquad\text{ with } a,b,c,d < s
\end{aligned}
\]
we conclude that $f(T_s) = 1$ for all $s \ge 6$ inductively.

But, if $s$ is sufficiently large, $T_s$ can be written as a sum of $k$ $\ptn$. So $f(T_s) = k$, which is a contradiction.


Thus, we can conclude that $f(2) = 2$, $f(3) = 3$, and $f(5) = 5$. Also, the above equality (\ref{eqn:T_s}) yields $f(T_s) = T_s$ for every $s$.

If $N$ is a sum of $k$ $\ptn$, then, clearly $f(N) = N$. Otherwise, we choose an integer $M$ such that $M > k+3$ and $\gcd(M,N)=1$. Then, since $M$ and $MN$ can be written as sums of $k$ $\ptn$, $M f(N) = f(M) \cdot f(N) = f(MN) = MN$. Thus, $f(N)=N$. The proof is completed.

\end{document}